\documentclass{amsart}
\usepackage{graphicx}
\usepackage{amssymb,amsmath,amsthm}
\newtheorem*{theorem}{\indent {\sc \textbf{Theorem}}}
\newtheorem{lemma}{\indent {\sc \textbf{Lemma}}}

\begin{document}

\title{On the Mean Values of the Function $\tau_k(n)$
in Sequences of Natural Numbers}%
\author{K. M.\'Eminyan}%
\address{Bauman State Technical University, Moscow}%
\email{eminyan@mail.ru}%

\keywords{sequence of natural numbers, trigonometric sum, number system of base q,
complex-valued function, inequality of the large sieve.}%

\begin{abstract}
We obtain an asymptotic formula for the mean value of the function $\tau_k(n)$, which is the
number of solutions of the equation $x_1 \cdots\, x_k=n$ in natural numbers $x_1,\ldots , x_k$ in some special sequences of natural numbers.
\end{abstract}
\maketitle
\begin{center}
{\bf{1. INTRODUCTION}}
\end{center}
\bigskip

Suppose that $q>1$, $n=c_{0}+c_{1}q+\ldots+c_{\nu}q^{\nu}$, $0\leqslant c_{0}, c_{1},\ldots, c_{\nu}<q$ --- the expansion of the natural number $n$ in the number system of base $q$. Then $S(n)=c_{0}+ c_{1}+\ldots+ c_{\nu}$.

In \cite{1}, Gelfond proved the following theorem: \emph{For the number $n$, $n\le x$, of integers, satisfying
the conditions
$$
n \equiv l\pmod m,\quad \sum_{b=0}^\nu c_b\equiv a \pmod p,\quad n=\sum_{b=0}^\nu c_bq^b, $$
where $q>1$, $p>1$, $m>1$; $l$ and $a$ are integers, and $(p, q-1) = 1$, the following asymptotic
formula holds:
$$
T_0(x)=\frac{x}{mp}+O(x^\lambda),\quad \lambda<1,
$$
where $\lambda$ is independent of $x$, $m$, $l$, $a$.}

In particular, if $p = q = 2$, then $\lambda = (\ln 3)/(2 \ln 2)$. In this particular case, the author obtained \cite{2} an asymptotic formula for the sum of the form
$$
\sum_{{n\leqslant X}\atop{\sum c_b\equiv a \pmod 2}}\tau(n),
$$
where $a$ is either 0 or 1.

The proof of Gelfond's theorem is based on his estimate of a trigonometric sum, which we cite in
Lemma 1.

In the present paper, we continue studies in this direction. The main result is the following theorem.

\begin{theorem}\label{t1}
Suppose that $k\geqslant 2$, $p>1$, and $\varepsilon >0$ is an arbitrarily small number. Suppose that $q$ is
a large natural number such that
$$\theta(q)=\frac{\ln(6(1+\ln q))}{\ln q}<\frac{1}{k},$$
here  $(q-1, p)=1$. Suppose that $a\in\mathbb{Z}$ is an arbitrary number.

Then the following asymptotic formula holds:
\[
\sum\limits_{n \leqslant x;S(n) \equiv a(\bmod p)} {\tau _k (n) = \frac{1}
{p}} \sum\limits_{n \leqslant x} {\tau _k (n) + O\left( {x^{1 - \frac{1}
{k} + \theta (q) + \varepsilon } } \right)}  + O\left( {x^{\lambda  + \varepsilon } } \right),
\]
where $\lambda \in(0, 1)$ depends only on $p$ and $q$.
\end{theorem}

In order to prove this theorem, in addition to Lemma 1, we need to estimate the integral of the modulus of the trigonometric sum given in Lemma 4.
\bigskip
\begin{center}
{\bf{2. THE LEMMAS}}
\end{center}
\bigskip
\begin{lemma}\label{l1}
Suppose that $\alpha$ is an arbitrary real number, $p>1$, $z$ is an integer, $(z,p)=1$, $(p,q-1)=1$,  $Q>1$, and
$$
S_Q(\alpha,z)=\sum_{n<q^Q}\exp\{\alpha n+\frac{z}{p}S(n)\}.
$$

Then the following inequality holds:
$$
|S_Q(\alpha,z)|\leqslant q^{\lambda Q},
$$
where $\lambda\in (0,1)$ depends only on $p$ and $q$.
\end{lemma}

Proof. For the proof of this lemma, see \cite{1}.

\begin{lemma}\label{l2}
Suppose that $T_0$ and $T\ge \delta>0$ are real numbers and, $f$ is a complex-valued continuous
function on the closed interval $[T_0,\ T_0+T]$ possessing a continuous derivative in the interval $(T_0,\ T_0+T)$. Suppose that $F$  is the set of real numbers from the interval $[T_0+\frac{\delta}{2},T_0+T-\frac{\delta}{2}]$ such that $|t-t'|>\delta$ for all distinct numbers $t$ and $t'$ from $F$. Then
$$
\sum_{t\in F}|f(t)|\leqslant \frac{1}{\delta}\int_{T_0}^{T_0+T}|f(t)|dt+\int_{T_0}^{T_0+T}|f'(t)|dt
$$
\end{lemma}

Proof. For the proof of this lemma, see \cite{3}.

\begin{lemma}\label{l3} {Suppose that $k\geqslant 2$, $\varepsilon >0$ is an arbitrarily small number, and $|f(n)|\leqslant 1$. Then the
following inequality holds:
\[
\left| {\sum\limits_{n \leqslant x} {\tau _k (n)f(n)} } \right| \ll x^\varepsilon  \sum\limits_{l \leqslant x^{1 - \frac{1}
{k}} } {\left| {\sum\limits_{a(l) < ln \leqslant x}^{} {f(ln)} } \right| + x^{1 - \frac{1}
{k} + \varepsilon } },
\]
where $a(l)$ depends only on $l$ and is less than $x$.}
\end{lemma}
\begin{proof} Suppose that $k=2$. Let us use the formula
\begin{equation}\label{Vin}
    \tau(n)=2\sum_{0<l<\sqrt{n}}\frac{1}{l}\sum_{y=0}^{l-1}e^{2\pi i\frac{ny}{l}}+\delta,
\end{equation}
where $\delta=1$ or $\delta=0$ depending on whether $n$ is the square of an integer or not (see \cite[p. 53]{Vin}).

We have
$$
\left|\sum_{n\le x}\tau(n)f(n)\right|\ll
\sum_{l<\sqrt{x}}\left|\sum_{l^2<ln\le x}f(ln)\right|+\sqrt{x}.
$$

Now let $k>2$. The following identity holds:
\[
S =  {\sum\limits_{n \leqslant x} {\tau _k (n)f(n)} }  = \sum\limits_{mn \leqslant x} {\tau _{k - 2} (m)\tau (n)f(mn)}  = S_1  + S_2,
\]
where
\[
\begin{gathered}
  S_1  = \sum\limits_{mn \leqslant x;n <\left( {mn} \right)^{2/k} } {\tau _{k - 2} (m)\tau (n)f(mn)},  \hfill \\
  S_2  = \sum\limits_{mn \leqslant x;m \leqslant \left( {mn} \right)^{1 - 2/k} } {\tau _{k - 2} (m)\tau (n)f(mn)};  \hfill \\
\end{gathered}
\]
here, by definition we set $\tau_{1}(m)$ identically equal to 1.

Let us estimate the sum $S_2$:
\begin{equation*}
\left| S_2 \right| \leqslant \sum\limits_{m \leqslant x^{1 - 2/k} } {\tau _{k - 2} (m)} \left| {\sum\limits_{m^{\frac{k}
{{k - 2}}}  \leqslant mn \leqslant x} {\tau (n)f(mn)} } \right|.
\end{equation*}

Again using formula (\ref{Vin}), we obtain
\[
\sum\limits_{m \leqslant x^{1 - 2/k} } {\tau _{k - 2} (m)} \left| {\sum\limits_{m^{\frac{k}
{{k - 2}}}  \leqslant mn \leqslant x} {\tau (n)f(mn)} } \right| \leqslant \]
\[\leqslant2\sum\limits_{m \leqslant x^{1 - 2/k} } {\tau _{k - 2} (m)} \left| {\sum\limits_{m^{\frac{k}
{{k - 2}}}  \leqslant mn \leqslant x} {\left( {\sum\limits_{d|n;d < \sqrt n } 1 } \right)f(mn)} } \right| + \sum\limits_{m \leqslant x^{1 - 2/k} } {\tau _{k - 2} (m)\sum\limits_{n^2  \leqslant \frac{x}
{m}} 1 }  \ll
\]
\[
\ll \sum\limits_{m \leqslant x^{1 - 2/k} } {\tau _{k - 2} (m)} \sum\limits_{d < \sqrt {\frac{x}
{m}} } {\left| {\sum_{\substack{m^{\frac{k}{{k - 2}}}  \leqslant dmn \leqslant x\\d^2m<dmn}}
{f(dmn)} } \right| + \sqrt x \sum\limits_{m \leqslant x^{1 - 2/k} } {\frac{{\tau _{k - 2} (m)}}
{{\sqrt m }}}}\ll
\]

\begin{equation}\label{n2}
  \ll \sum\limits_{l < x^{1 - 1/k} } {\left( {\sum\limits_{md = l;m \leqslant x^{1 - 2/k} ;d < \sqrt {\frac{x}
{m}} } {\tau _{k - 2} (m)} } \right)} \left|
{\sum\limits_{\substack{m^{\frac{k}{k - 2}}  \leqslant ln \leqslant x\\dl<ln}}
{f(ln)} } \right| + x^{1 - \frac{1}
{k} + \varepsilon }.
\end{equation}

Suppose that $l=dm$, where $d<\sqrt{\frac{x}{m}}$, $m\leqslant x^{1-2/k}$, and $a(l,m)=\max([m^{\frac{k}{k - 2}}]+1, dl)$.

Let $a(l)$ --- be the value of $a(l,m)$, for which the modulus of the sum
$$
\sum\limits_{{a(l,m) <ln \leqslant x}}
{f(ln)}
$$
is maximal. Then
\[
|S_2|\ll\sum\limits_{l < x^{1 - 1/k} }{\left( {\sum\limits_{m\mid l;m \leqslant x^{1 - 2/k} ;\frac{l}{m} < \sqrt {\frac{x}
{m}} } {\tau _{k - 2} (m)} } \right)} \left| {\sum\limits_{a(l)< ln \leqslant x} {f(ln)} } \right|+ x^{1 - \frac{1}
{k} + \varepsilon } \ll
\]
\[
\ll x^\varepsilon \sum\limits_{l < x^{1 - 1/k} }  \left| {\sum\limits_{a(l)  < ln \leqslant x} {f(ln)} } \right|+ x^{1 - \frac{1}
{k} + \varepsilon }.
\]

It remains to estimate the sum $S_1$.

If $k=3$, then $\tau_{k-2}(m)=1$, $\frac{2}{k}=1-\frac{1}{k}$, therefore,
$$
|S'_1|\leqslant \sum_{l<x^{1-1/k}}\tau(l)|\sum_{l^{\frac{k}{2}}<lm\leqslant x}f(lm)|\ll x^\varepsilon\sum_{l<x^{1-1/k}}|\sum_{l^{\frac{k}{2}}<lm\leqslant x}f(lm)|.
$$

If $k=4$, then  $\frac{2}{k}=1-\frac{2}{k}$, therefore,
$$
|S'_1|\leqslant \sum_{l<x^{1-2/k}}\tau(l)|\sum_{l^{\frac{k}{2}}<lm\leqslant x}\tau(m)f(lm)|.
$$
The last sum is estimated in the same way as $S_2$.

Suppose that $k\geqslant 5$. The following identity holds:
$$
S_1=S_1'+S_1'',
$$
where
$$
S_1'=\sum_{n<x^{2/k}}\tau(n)\sum_{\substack{n^{\frac{k}{2}}<m_1m_2n\leqslant x\\ m_2<(m_1m_2n)^{2/k}}}\tau_{k-4}(m_1)\tau(m_2)f(m_1m_2n),
$$
$$
S_1''=\sum_{n<x^{2/k}}\tau(n)\sum_{\substack{n^{\frac{k}{2}}<m_1m_2n\leqslant x\\ m_1n\leqslant(m_1m_2n)^{1-2/k}}}\tau_{k-4}(m_1)\tau(m_2)f(m_1m_2n).
$$

For $S_1''$ we have the inequality
$$
|S_1''|\leqslant \sum_{l\leqslant x^{1-2/k}}(\sum_{\substack{n\mid l\\n<x^{2/k}}}\left(\tau(n)\tau_{k-4}\left(\frac{l}{n}\right)\right)
\left|\sum_{\substack{n^{k/2}<lm_2\leqslant x\\l^{\frac{k}{k-2}}<lm_2}}\tau(m_2)f(lm_2)\right|.
$$

The last sum is estimated in the same way as $S_2$.

Consider the sum
$$
S_1'=\sum_{n<x^{2/k}}\tau(n)\sum_{m_2<x^{2/k}}\tau(m_2)\sum_{\substack{n^{\frac{k}{2}}<m_1m_2n\leqslant x\\ m_2^{k/2}<(m_1m_2n)}}\tau_{k-4}(m_1)f(m_1m_2n).
$$

Note that, as a result of the transformation $\tau_{k-2}(m)$ is replaced by $\tau_{k-4}(m_1)$.
If $k\neq 5,\, 6$ we shall repeat this transformation until $\tau_{k-4}(m_1)$ is replaced by $\tau_1(m_1)$ or $\tau(m_1)$ (depending on whether $k$ is even or
odd).

If $k$ is an odd number, then, as a result, we obtain the inequality
$$
|S_1'|\ll x^\varepsilon \sum\limits_{l < x^{1 - 1/k} }  \left| {\sum\limits_{a(l)  < lm \leqslant x} {f(lm)} } \right|,
$$
where $a(l)$ is a number less than $x$ and, possibly, not coinciding with the number $a(l)$ which appears in the estimate of $S_2$.

But if $k$ is an even number, then we obtain the inequality
$$
|S_1'|\ll x^{\frac{\varepsilon}{2}}\sum\limits_{l < x^{1 - 2/k} } \left| {\sum\limits_{b(l)  < lm \leqslant x} {\tau(m)f(lm)} } \right|,
$$
where $b(l)<x$. The last sum is estimated in the same way as $S_2$.

Lemma 3 is proved.
\end{proof}

\begin{lemma}\label{l4}
Suppose that $Q>1$,
$$
 S_{Q}(\alpha, z)=\sum\limits_{n<q^{Q}}e^{2\pi i (\alpha n+\frac{z}{p}S(n))}.
$$

Then the following inequality holds:
$$
\int_{0}^{1}|S_{Q}(\alpha, z)|\,d\alpha \leqslant q^{Q\theta},
$$
where
$\theta= \frac{\ln(6(1+\ln q))}{\ln q}$.
\end{lemma}

\begin{proof}  We have the identity
$$
 S_{Q}(\alpha, z)=\prod\limits_{r=0}^{Q-1}\sum\limits_{n=0}^{q-1}e^{2\pi i n (\alpha q^{r}+\frac{z}{p})}
$$
Let us divide the interval of integration into $q$ equal parts:
$$
\int_{0}^{1}|S_{Q}(\alpha, z)|\,d\alpha =\sum\limits_{j=0}^{q-1}\int_{j/q}^{(j+1)/q}|S_{Q}(\alpha, z)|\,d\alpha.
$$

In all the resulting integrals, let us make the change of variables
$$
\alpha =\frac{x+j}{q}, {~~~} j=0,1,\ldots,q-1.
$$
Then we obtain
$$
\int_{0}^{1}|S_{Q}(\alpha, z)|\,d\alpha =\frac{1}{q}\sum\limits_{j=0}^{q-1}\int_{0}^{1}\left|S_{Q}\left(\frac{x+j}{q}, z\right)\right|\,dx.
$$
Consider $S_{Q}\left(\frac{x+j}{q}, z\right)$. Extracting the first factor in the product, we obtain the identity
\[
S_Q \left( {\frac{{x + j}}
{q},z} \right) = \sum\limits_{n = 0}^{q - 1} {e^{2\pi in\left( {\frac{{x + j}}
{q} + \frac{z}
{p}} \right)} } \prod\limits_{r = 1}^{Q - 1} \sum\limits_{n=0}^{q-1}{e^{2\pi in\left( {\frac{{x + j}}
{q}q^r  + \frac{z}
{p}} \right)} }  =\] \[= \sum\limits_{n = 0}^{q - 1} {e^{2\pi in\left( {\frac{{x + j}}
{q} + \frac{z}
{p}} \right)} } S_{Q-1} \left( {x,z} \right).
\]

Suppose that
\[
h_j (x) = \left| {\sum\limits_{n = 0}^{q - 1} {e^{2\pi in\left( {\frac{{x + j}}
{q} + \frac{z}
{p}} \right)} } } \right|
.\]

We have the inequality
\[
\int_0^1 {\left| {S_Q \left( \alpha  \right)} \right|} d\alpha
\leqslant \int_0^1 {\frac{1} {q}} \sum\limits_{j = 0}^{q - 1} {h_j
(x)} \left| {S_{Q-1} (x)} \right|dx.
\]
Let us obtain a uniform (in $x$) estimate of the sum
\[
\sum\limits_{j = 0}^{q - 1} {h_j (x)}.
\]
The inequality
$
h_j (x)\leqslant q
$
is trivial.

Further, if
$\frac{z}{p}+\frac{x+j}{q} \notin \mathbb{Z}$,
then
\[
h_j (x) = \left| {\frac{{1 - e^{2\pi iq\left( {\frac{{x + j}}
{q} + \frac{z}
{p}} \right)} }}
{{1 - e^{2\pi i\left( {\frac{{x + j}}
{q} + \frac{z}
{p}} \right)} }}} \right| \leqslant \frac{1}
{{\left| {\sin \pi \left( {\frac{{x + j}}
{q} + \frac{z}
{p}} \right)} \right|}} \leqslant \frac{1}
{{2\parallel{{\frac{{x + j}}
{q} + \frac{z}{p}\parallel}}}},
\]
where $\|x\|$ is the distance from $x$ to the nearest integer. Thus,
\[
h_j (x) \leqslant \min \left( {q,\frac{1}
{{2\left| {\left| {\frac{{x + j}}
{q} + \frac{z}
{p}} \right|} \right|}}} \right).
\]
Let us estimate the sum
\[
\frac{1}
{q}\sum\limits_{j = 0}^{q - 1} {\min \left( {q,\frac{1}
{{2\left| {\left| {\frac{{x + j}}
{q} + \frac{z}
{p}} \right|} \right|}}} \right)}  \leqslant \frac{6}
{q}\left( {q + q\ln q} \right) = 6(1 + \ln q)
\]
(see, for example, \cite{AAK}).

We have obtained the inequality
\[
\int_0^1 {\left| {S_Q \left( \alpha  \right)} \right|} d\alpha
\leqslant 6(1 + \ln q)\int_0^1 {\left| {S_{Q - 1} \left( \alpha
\right)} \right|} d\alpha,
\]
valid for any $Q$, greater than 1. It follows that
\[
\int_0^1 {\left| {S_Q \left( \alpha  \right)} \right|} d\alpha
\leqslant \left( {6(1 + \ln q)} \right)^{Q - 1} \int_0^1 {\left|
{S_1 \left( \alpha  \right)} \right|} d\alpha  \leqslant
\]
\[
 \leqslant \left( {6(1 + \ln q)} \right)^{Q - 1} \int_0^1 {\min \left( {q,\frac{1}
{{\| {\alpha  + \frac{z}
{p}} \|}}} \right)} d\alpha.
\]

Let us estimate the last integral. Since the function $\|x\|$ is periodic with period 1 and even, it follows
that
$$
\int_0^1 {\min \left( {q,\frac{1} {{\| {\alpha  + \frac{z} {p}}
\|}}} \right)}
d\alpha=2\int_0^{1/2}\min\left(q,\frac{1}{\|\alpha\|}\right)d\alpha\leqslant
2+2\int^{1/2}_{1/q}\frac{d\alpha}{\alpha}<2+2\log q.
$$

This yields
\[
\int_0^1 {\left| {S_Q \left( \alpha  \right)} \right|} d\alpha
\ll\left( {6(1 + \ln q)} \right)^Q  = q^{Q\frac{{\ln \left( {6(1 +
\ln q)} \right)}} {{\ln q}}},\] which proves the assertion.
\end{proof}
\newpage
\bigskip
\bigskip
\bigskip
\bigskip
\bigskip
\bigskip
\begin{center}
\textbf{3. PROOF OF THE THEOREM \ref{t1}}
\end{center}
\bigskip

1. \emph{Preparing for the application of a large sieve}. We have the chain of equalities
$$
\sum\limits_{n \leqslant x;S(n) \equiv a(\bmod p)}\tau_k(n)=\sum_{n\leqslant x}\frac{1}{p}\sum_{z=0}^{p-1}\tau_k(n)e^{2\pi i\frac{z(S(n)-a)}{p}}=
$$
$$
=\frac{1}{p}\sum\limits_{n \leqslant x}\tau_k(n)+\frac{1}{p}\sum_{z=1}^{p-1}e^{-2\pi i \frac{az}{p}}\sum_{n\leqslant x}\tau_k(n)e^{2\pi i\frac{zS(n)}{p}}.
$$

Now, in order to prove the theorem, it suffices to estimate the sum
$$
\sum_{n\leqslant x}\tau_k(n)e^{2\pi i\frac{zS(n)}{p}}.
$$
for any noninteger $\frac{z}{p}$.  We apply Lemma \ref{l3} to this sum, obtaining:
$$
\left|\sum_{n\leqslant x}\tau_k(n)e^{2\pi i\frac{zS(n)}{p}}\right|\ll Wx^\varepsilon+x^{1-\frac{1}{k}+\varepsilon},
$$
where
$$
W = \sum\limits_{l \leqslant x^{1 - \frac{1}
{k}} } {\left|\sum\limits_{a(l) < ln \leqslant x}
e^{2\pi i(\frac{z}{p}S(ln))}  \right|},
$$
$a(l)<x$. Suppose that $H\in \mathbb{N}$, $x\in[q^{H-1}, q^{H})$, and $H\geqslant H_{0}>1$. Then
\[
W \leqslant \sum\limits_{l \leqslant x^{1 - \frac{1}
{k}} } {\frac{1}
{l}\sum\limits_{b = 1}^{l } {\left| {\sum\limits_{a(l) < n \leqslant x}{e^{2\pi i\left( {\frac{b}
{l}n + \frac{z}
{p}S(n)} \right)} } } \right|} }  =
\]
\[
 = \sum\limits_{l \leqslant x^{1 - \frac{1}
{k}} } {\frac{1} {l}\sum\limits_{b = 1}^{l } {\left|
{\sum\limits_{n < q^H }^{} {e^{2\pi i\left( {\frac{b} {l}n +
\frac{z} {p}S(n)} \right)} } \sum\limits_{a(l) < n_1  \leqslant x}
{\int_0^1 {e^{2\pi iy(n_1  - n)} } } } dy\right|}}\leqslant
\]
\[
 \leqslant \int_0^1 {\sum\limits_{l < x^{1 - 1/k} } {\frac{1}
{l}} \sum\limits_{b = 1}^{l} {\left| {\sum\limits_{a(l) < n_1
\leqslant x}  {e^{2\pi iyn_1 } } } \right|\left| {\sum\limits_{n <
q^H }^{} {e^{2\pi i\left( {\left( {\frac{b} {l} - y} \right)n +
\frac{z} {p}S(n)} \right)} } } \right|dy} }\leqslant
\]
\[
 \leqslant \int_0^1 {\min \left( {x,\frac{1}
{{\left\| y \right\|}}} \right)\sum\limits_{l < x^{1 - 1/k} } {\frac{1}
{l}} \sum\limits_{b = 1}^{l} {\left| {\sum\limits_{n < q^H }^{} {e^{2\pi i\left( {\left( {\frac{b}
{l} - y} \right)n + \frac{z}
{p}S(n)} \right)} } } \right|dy} }.
\]
Let us introduce the notation
\[
S_H \left( {\alpha ,z} \right) = \sum\limits_{n < q^H } {e^{2\pi i\left( {\alpha n + \frac{z}
{p}S(n)} \right)} } .
\]

The following identity holds:
\begin{equation}\label{***}
S_H \left( {\alpha ,z} \right) = \prod\limits_{l = 0}^{H - 1} {\sum\limits_{u = 0}^{q - 1} {e^{2\pi iu ( \alpha q^l  + z/p)}}}.
\end{equation}
Let us rewrite the sum

\[
\sum\limits_{l< x^{1 - 1/k} } {\frac{1}
{l}} \sum\limits_{b = 1}^{l} {\left| {S_H \left( {\frac{b}
{l} - y,z} \right)} \right| = } \sum\limits_{d < x^{1 - 1/k} } {\sum\limits_{l \leqslant x^{1 - 1/k} } {\frac{1}
{l}} } \sum\limits_{\substack{b = 1\\\left( {b,l} \right) = d}}^{l} {\left| {S_H \left( {\frac{b}
{l} - y,z} \right)} \right| = }
\]
\[
 = \sum\limits_{d < x^{1 - 1/k} } {\frac{1}
{d}\sum\limits_{l_1  \leqslant \frac{{x^{1 - 1/k} }}
{d}} {\frac{1}
{{l_1 }}} } {\sum\limits_{b_1  = 1}^{l_1}}{^{'}} { \left| {S_H \left( {\frac{{b_1 }}
{{l_1 }} - y,z} \right)} \right| \ll }
\]
\[
 \ll \log x{\sum\limits_{l_{1} < x^{1 - 1/k} } {\frac{1}
{{l_1 }}}}{ \sum\limits_{b_1  = 1}^{l_1 }}{^{'}} {\left| {S_H \left( {\frac{{b_1 }}
{{l_1 }} - y,z} \right)} \right|},
\]
where the prime means that $(b_{1}, l_{1})=1$.

Further,
\[
\sum\limits_{l_1  < x^{1 - 1/k} } {\frac{1}
{l_1}} \mathop{\sum_{b_1  = 1}^{l_1 }}{^{'}} {\left| {S_H \left( {\frac{{b_1 }}
{{l_1 }} - y,z} \right)} \right|}  \ll\]
\[\ll \sum\limits_{r \leqslant \left( {1 - \frac{1}
{k}} \right)\log _q x + 1} {q^{ - r+1} } \sum\limits_{q^{r - 1}  \leqslant l_1  < q^r } {\sum\limits_{b_1  = 1}^{l_1 }{^{'}} { \left| {S_H \left( {\frac{{b_1 }}
{{l_1 }} - y,z} \right)} \right| \leqslant } }
\]
\[
 \leqslant \sum\limits_{r \leqslant \frac{H}
{2}} {q^{ - r+1} } \sum\limits_{q^{r - 1}  \leqslant l_1  < q^r } {\sum\limits_{b_1  = 1}^{l_1 } {{^{'}} \left| {S_H \left( {\frac{{b_1 }}
{{l_1 }} - y,z} \right)} \right| + } }\] \[+ \sum\limits_{\frac{H}
{2} < r \leqslant \left( {1 - \frac{1}
{k}} \right)\log _q x + 1} {q^{ - r+1} } \sum\limits_{q^{r - 1}  \leqslant l_1  < q^r } {\sum\limits_{b_1  = 1}^{l_1 } {{^{'}} \left| {S_H \left( {\frac{{b_1 }}
{{l_1 }} - y,z} \right)} \right|} }.
\]

For $r<\frac{H}{2}$ we use the identity
\[
\left| {S_H \left( {\frac{{b_1 }}
{{q_1 }} - y,z} \right)} \right| = \left| {S_{2r} \left ( {\frac{{b_1 }}
{{q_1 }} - y,z} \right)} \right|\,\left| {S_{H - 2r}( ( {\frac{{b_1 }}
{{q_1 }} - y)q^{2r},z)} } \right|,
\]
which immediately follows from (\ref{***}), and also use Lemma 1
\[
\left| {S_{H - 2r} \left( {\left( {\frac{{b_1 }}
{{q_1 }} - y} \right)q^{2r} ,z} \right)} \right| \ll q^{\left( {H - 2r} \right)\lambda +1}  \ll x^\lambda  q^{ - 2r\lambda+1},
\]
where $0<\lambda <1$, which is valid for $\frac{z}{p}\notin \mathbb{Z}$.

Thus,
\[
W \ll x^\lambda\log x  \sum\limits_{r \leqslant \frac{H}
{2}} {q^{ - r - 2r\lambda+1 } } \int_0^1\min \left(x,\frac{1}{\|y\|}\right)\sum\limits_{q^{r - 1}  \leqslant l_1  < q^r } {\sum\limits_{b_1  = 1}^{l_1 } { \left| {S_{2r} \left( {\frac{{b_1 }}
{{l_1 }} - y,z} \right)} \right|} }dy +
\]
\[
 + \log x \int_0^1\min \left(x,\frac{1}{\|y\|}\right)\sum\limits_{\frac{H}
{2} < r \leqslant \left( {1 - \frac{1}
{k}} \right)\log _q x + 1} {q^{ - r+1} } \sum\limits_{q^{r - 1}  \leqslant l_1  < q^r }{\sum\limits_{b_1  = 1}^{l_1 } {{^{'}} \left| {S_{H} \left( {\frac{{b_1 }}
{{l_1 }} - y,z} \right)} \right|} }dy.
\]

2. \emph{Termination of the proof}. Now we have all that is needed to apply the inequality of the large sieve
(Lemma 2). We have

\[
W\ll x^\lambda \log x \int_0^1\min
\left(x,\frac{1}{\|y\|}\right)dy\sum\limits_{r \leqslant \frac{H} {2}} {q^{r
- 2r\lambda +1} } \int_0^1 {\left| {S_{2r} \left( {\alpha-y,z}
\right)} \right|} d\alpha  +
\]
\[
 + \log x \int_0^1\min \left(x,\frac{1}{\|y\|}\right)dy\sum\limits_{\frac{H}
{2} < r \leqslant \left( {1 - \frac{1} {k}} \right)\log _q x + 1}
{q^{ r+1} } \int_0^1 {\left| {S_{H} \left( {\alpha  - y,z}
\right)} \right|} d\alpha .
\]

Since the functions $S_{2r}(\alpha, z)$ and $S_{H}(\alpha, z)$ are periodic with period, by Lemma \ref{l4} we have
\[
\int_0^1 {\left| {S_{2r} \left( {\alpha  - y,z} \right)} \right|}
d\alpha  = \int_0^1 {\left| {S_{2r} \left( {\alpha ,z} \right)}
\right|} d\alpha  \ll q^{2r\theta} ,
\]
\[
\int_0^1 {\left| {S_H \left( {\alpha  - y,z} \right)} \right|}
d\alpha  = \int_0^1 {\left| {S_H \left( {\alpha ,z} \right)}
\right|} d\alpha  \ll q^{H\theta }  \ll x^\theta;
\]
therefore,
\[
W\ll x^\lambda \log^2 x  \sum\limits_{r \leqslant \frac{H}
{2}} {q^{r\left( {1 - 2\lambda  + 2\theta } \right)+1} }  + \log^2 x\sum\limits_{\frac{H}
{2} < r \leqslant \left( {1 - \frac{1}
{k}} \right)\log _q x + 1} {q^{  r+1} } x^\theta \ll
\]
\[
\ll \left(x^\lambda  \sum\limits_{r \leqslant \frac{H}
{2}} {q^{r\left( {1 - 2\lambda  + 2\theta } \right)+1} }  + q x^{1 - \frac{1}
{k} + \theta }\right)\log^2 x.
\]

Consider the following two cases:

1){~~~} if  $1-2\lambda+2\theta <0$; then $\sum\limits_{r \leqslant \frac{H}
{2}} {q^{r\left( {1 - 2\lambda  + 2\theta } \right)} }  \ll 1;$

2){~~~} if $1-2\lambda+2\theta \geqslant 0$; then $\sum\limits_{r \leqslant \frac{H}
{2}} {q^{r\left( {1 - 2\lambda  + 2\theta } \right)} }  \ll q^{\frac{H}
{2}\left( {1 - 2\lambda  + 2\theta } \right)}  \ll x^{\frac{1}{2} - \lambda  + \theta }$.

Thus,

\[
\left| W \right| \ll \left( {x^\lambda   + x^{1 - \frac{1}
{k} + \theta } } \right)\ln ^2 x\ll \left( {x^\lambda   + x^{1 - \frac{1}
{k} + \theta } } \right)\ln ^2 x.
\]
(the multiplier $q$ in the final estimate is not written, because we assume that the parameter $q$ is
nonincreasing with the growth of the main parameter $x$).

Theorem  is proved.

\end{document}